\documentclass[10pt]{article}

\usepackage{amsmath,amsfonts,amssymb}
\usepackage{epsf}
\RequirePackage{amsfonts,amssymb,amsmath,amscd,amsthm}
\usepackage{epsfig,tabularx,graphicx,float}
\usepackage{graphicx}
\usepackage[a4paper]{geometry}

\usepackage{parskip}
\usepackage[utf8]{inputenc}
\usepackage{times}
\usepackage{float}
\usepackage{multicol}
\usepackage{multirow}

\usepackage{float}
\newtheorem{theorem}{Theorem}[section]
\newtheorem{proposition}[theorem]{Proposition}
\newtheorem{remark}[theorem]{Remark}
\newtheorem{definition}[theorem]{Definition}
\newtheorem{corollary}[theorem]{Corollary}

\newtheorem{lem}[theorem]{Lemma}
\begin{document}

\begin{flushleft}
Prasadini Mahapatra \footnote{\textbf{Prasadini Mahapatra}, Department of Mathematics, National Institute of Technology Rourkela, Odisha, India} and Divya Singh \footnote{\textbf{Corresponding Author: Divya Singh}, Department of Mathematics, National Institute of Technology Rourkela, Odisha, India}
\end{flushleft}

\begin{flushleft}
\textbf{\Large Properties of translates of frames on Vilenkin group}\\[1 cm]
\end{flushleft}

\noindent \textbf{Abstract:} 
We study the properties based on the space generated by the translates of square
integrable function using C-bracket on Vilenkin group. We give the necessary and
sufficient condition for frame sequences by the family of translates of the function.



\section{Introduction}\label{sec1}
In recent years several works related to generalizations and extensions of wavelets have been introduced. Some fundamental ideas from wavelet theory, such as multiresolution analysis (MRA), have appeared in very different contexts. The authors like Schipp, Wade, Simon, Golubov, Efimov, Skvortsov, Maqusi, Siddiqi, Beauchamp developed the theory related to Walsh analysis. In 1923, J. Walsh introduced Walsh functions as a linear combination of Haar functions. In the 1940s Gelfand recognized that Walsh functions are identified with characters of the Cantor dyadic group. There is a lot of work done on locally compact abelian group. For every locally compact abelian group, there exists a Haar measure, which is not identically zero but unique up to a multiplicative constant. N. J. Fine and N. Ya Vilenkin independently determined that the Walsh system is the group of characters of the Cantor dyadic group. A large class of locally compact abelian groups was introduced by  N. Ya Vilenkin, called Vilenkin groups. Cantor dyadic group is its particular case.
Lang \cite{lang1} determined the compactly supported orthogonal wavelets on the locally compact Cantor dyadic group. Shift operators such as dilation operator, multiresolution analysis(MRA) were built and some necessary regularity conditions for wavelets and sufficient conditions on scaling filters were given to occur orthonormal wavelets. The generalised Walsh functions form an orthonormal system in the Vilenkin group $G$ and mask of refinable equation is given in terms of these generalised Walsh functions. Refinable equation gives refinable function which generates MRA and hence wavelets, if the mask satisfies certain conditions. Necessary and sufficient conditions were given over the mask of scaling function $\phi$ in terms of modified Cohen's condition and blocked sets such that $\phi$ generates an MRA. Lang \cite{lang2} examined that the necessary and sufficient conditions were given on a trigonometric polynomial scaling filter resulting in a multiresolution analysis.

In the last decade, several authors such as Aldroubi, Benedetto, Bownik, de Boor, DeVore, Li, Ron, Rzeszotnik, Shen, Weiss and Wilson, Behera have been studied the shift invariant subspaces for $L^{2}$ space of $\mathbb{R}^{d}$ by different aspects. This theory plays an important role in many areas, especially in the theory of wavelets, and multiresolution
analysis. Shift-invariant spaces are very important in applications and the theory had a great development in the last twenty years, mainly in approximation theory, sampling, wavelets, and frames. In particular, they serve as models in many problems in signal and image processing. The concept of frames was first introduced in 1952 by R. Duffin and A. C. Schaefer in their work on nonharmonic analysis. 

This paper is organized as follows. It consists of four main sections. In Section 2, we study the properties on C-bracket and generating space on the Vilenkin group. Section 3 contains the characterization of translates of frames and section 4 consists of the characterization of generalized dual frames.

\section{Notation and Preliminaries}\label{sec2}

The Vilenkin group $ G $ is defined as the group of sequences
$$
x=(x_{j})=(..., 0, 0, x_{k}, x_{k+1}, x_{k+2},...),
$$
where $x_{j}\in \lbrace 0, 1, ..., p-1\rbrace $ for $ j \in \mathbb{Z}$ and $ x_{j}=0 $, for $ j < k = k(x)$. The group operation on $ G $, denoted by $ \oplus $, is defined as coordinatewise addition modulo $ p $:
$$
(z_{j})=(x_{j})\oplus (y_{j})\Leftrightarrow z_{j}=x_{j} + y_{j}(\text{mod}\: p), \text{ for } j \in \mathbb{Z}.
$$
Let
$$
U_{l}=\lbrace(x_{j}) \in G : x_{j}=0\: \text{for}\: j\leq l\rbrace ,\qquad l\in \mathbb{Z},
$$
be a system of neighbourhoods of zero in $G$. In case of topological groups if we know neighbourhood system $\lbrace U_l \rbrace_{l \in \mathbb Z}$ of zero, then we can determine neighbourhood system of every point $x=(x_j) \in G$ given by $\lbrace U_l \oplus x \rbrace_{l \in \mathbb Z}$, which in turn generates a topology on $G$.

Let $ U=U_{0} $ and $ \ominus $ denotes the inverse operation of $ \oplus $. The Lebesgue spaces $ L^{q}(G),\: 1\leq q\leq \infty $, are defined by the Haar measure $ \mu $ on Borel's subsets of $ G $ normalized by $ \mu(U)=1 $.

The group dual to $ G $ is denoted by $ G^{*} $ and consists of all sequences of the form
$$
\omega=(\omega_{j})=(...,0,0,\omega_{k},\omega_{k+1},\omega_{k+2},...),
$$
where $ \omega_{j} \in \lbrace 0,1,...,p-1\rbrace $, for $ j \in \mathbb{Z} $ and $ \omega_j=0 $, for $ j<k=k(\omega) $. The operations of addition and subtraction, the neighbourhoods $\lbrace U_{l}^{*}\rbrace $ and the Haar measure $ \mu^{*} $ for $ G^{*} $ are defined as above for $ G $. Each character on $ G $ is defined as
$$
\chi(x,\omega)=\exp\bigg(\frac{2\pi i}{p}\sum_{j\in \mathbb{Z}}{x_{j}w_{1-j}}\bigg),\quad x \in G,
$$
for some $ \omega \in G^{*} $.

Let $ H = \lbrace (x_{j}) \in G\: |\: x_{j} =0 \; \text{ for } j>0\rbrace $ be a discrete subgroup in $ G $  and $A$ be an automorphism on $G$ defined by $ (Ax)_j=x_{j+1} $, for $x=(x_j) \in G$. From the definition of annihilator and above definition of character $\chi$, it follows that the annihilator $ H^{\perp} $ of the subgroup $ H $ consists of all sequences $ (\omega_j) \in G^{*} $ which satisfy $ \omega_j=0 $ for $ j>0 $.

Let $ \lambda:G\longrightarrow \mathbb{R}_+ $ be defined by
\begin{equation*}
\lambda(x)=\sum_{j \in \mathbb{Z}}{x_{j} p^{-j}}, \qquad x=(x_j) \in G.
\end{equation*}
It is obvious that the image of $ H $ under $ \lambda $ is the set of non-negative integers $ \mathbb{Z}_{+} $. For every $ \alpha \in \mathbb{Z}_{+} $, let $ h_{[\alpha]} $ denote the element of $ H $ such that $ \lambda(h_{[\alpha]})=\alpha $. For $ G^{*} $, the map $ \lambda^{*}:G^{*}\longrightarrow \mathbb{R}_{+} $, the automorphism $ B \in \text{Aut } G^{*} $, the subgroup $ U^{*} $ and the elements $ \omega_{[\alpha]} $ of $ H^{\perp} $ are defined similar to $ \lambda$, $A$, $U$ and $h_{[\alpha]}$, respectively. 

The generalised Walsh functions for $ G $ are defined by
\begin{equation*}
W_\alpha(x)=\chi(x,\omega_{[\alpha]}),\qquad \alpha \in \mathbb{Z}_+, x \in G.
\end{equation*}
These functions form an orthogonal set for $L^2(U)$, that is,
\begin{equation*}
\int_{U}{W_{\alpha}(x) \overline{W_{\beta}(x)}d\mu(x)}=\delta_{\alpha,\beta},\qquad \alpha,\beta \in \mathbb{Z}_{+},
\end{equation*}
where $ \delta_{\alpha,\beta} $ is the Kronecker delta. The system $ {W_{\alpha}} $ is complete in $ L^{2}(U) $. The corresponding system for $ G^{*} $ is defined by
\begin{equation*}
W_{\alpha}^{*}(\omega)=\chi(h_{[\alpha]},\omega),\qquad \alpha \in \mathbb{Z}_{+}, \omega \in G^{*}.
\end{equation*}
The system $\lbrace W_{\alpha}^{*} \rbrace$ is an orthonormal basis of $ L^{2}(U^{*}) $.

For positive integers $ n $ and $ \alpha $,
\begin{equation*}
U_{n,\alpha}=A^{-n}(h_{[\alpha]}) \oplus A^{-n}(U).
\end{equation*}

\begin{definition}
A sequence $\{f_{j}:j\in J\}$ in a Hilbert space $H$ is called a Bessel sequence if there exists $B>0$ such that 
$$
\sum_{j \in J} \vert \langle f,f_{j} \rangle \vert ^{2}\leq B \Vert f \Vert^{2}, \quad \forall f \in H.
$$
\end{definition}

\begin{definition}
A sequence $\{f_{j}:j\in J\}$ in a Hilbert space $H$ is called a \emph{frame} if there exist constants $a$ and $b$, $0 <a \leq b < \infty$, such that
$$
a \Vert f \Vert ^{2} \leq \sum_{j \in J} \vert \langle f,f_{j} \rangle \vert ^{2}\leq b \Vert f \Vert^{2}, \quad \forall f \in H.
$$ 
\end{definition}
If $a=b$, then $(f_j)$ is called a \emph{tight frame}, and it is called a \emph{Parseval frame} if $a=b=1$.

\section{Translates of frames}\label{sec3}
The principal shift invariant space generated by $\psi$ is denoted as $\langle \psi \rangle$, defined by $\langle \psi \rangle=\overline{\text{span}}\{\psi(. \ominus h):h \in H\}$.  
 
In general, a closed subspace $V \subseteq L^{2}(G)$ is called shift invariant if and only if $T_{h}V \subset V$, for all $h \in H$. 
For $\psi \in L^2(G)$, the periodization function is defined by 
\begin{eqnarray} 
P_{\psi}(\omega)=\sum_{h \in H^{\bot}} \vert \hat{\psi}(\omega \oplus h) \vert ^{2}.
\end{eqnarray}

For $f,g \in L^{2}(G)$, the \emph{$C$-bracket} is defined as
\begin{equation}\label{110}
[f,g](x):=\sum_{h \in H} f(x \oplus h)\overline{g(x \oplus h)}.
\end{equation}

Clearly, $P_{\psi}=[\hat{\psi},\hat{\psi}]$.

Define the map $J_{\psi}:M_{\psi} \rightarrow L^{2}(G)$ by $J_{\psi}m=(m\hat{\psi})^{\check{}}$, for $m \in M_{\psi}$. Then we have 

\begin{theorem}(\cite{kg1})
$J_{\psi}$ is an isometry between $M_{\psi}$ and $\langle \psi \rangle$.	
\end{theorem}

\begin{theorem}(\cite{kg1})\label{charPSI1}
Let $\phi \in L^2(G)$. Then $f \in \langle \phi \rangle$ iff $\hat{f}(\xi)=m(\xi)\hat{\phi}(\xi)$, for some $m \in M_{\psi}$.
\end{theorem}

\begin{theorem}
There exists a cononical dual $\tilde{\psi}$ of $\psi$ that belongs to $\langle \psi \rangle$ iff $\frac{1}{P_{\psi}}$ belongs to $L^{1}(U^{*})$. In this case $\widehat{\tilde{\psi}}=\frac{1}{P_{\psi}}\hat{\psi}$.
\end{theorem}

\begin{proof}
For $\tilde{\psi} \in \langle \psi \rangle$, it follows from Theorem \ref{charPSI1} that there exists $m \in M_{\psi}$ such that $\hat{\tilde{\psi}}=m \hat{\psi}$. Since $\tilde{\psi}$ is canonical dual of $\psi$, 
$$\langle T_{h}\psi, \tilde{\psi} \rangle =\delta_{h,0}, \forall h \in H.$$

Now, $[\hat{\psi}, \hat{\tilde{\psi}}]=\overline{m}P_{\psi}$ in $L^{1}(U^{*})$, and we have $\overline{m}(\omega)P_{\psi}(\omega)=1$ a.e. Thus $m=\overline{m}=\frac{1}{P_{\psi}}$. The result follows as $\frac{1}{P_{\psi}}\in M_{\psi}$ if and only if $\frac{1}{P_{\psi}} \in L^{1}(U^{*})$.
\end{proof}

Let $\Lambda:\langle \phi \rangle \rightarrow l^{2}(\mathbb N_0)$ be the analysis operator. The adjoint of $\Lambda$, denoted by $\Lambda^*$ is called the synthesis operator. Let  $(a_{h})_{h \in H}$ be the canonical basis of $l^{2}(\mathbb N_0)$, where $\mathbb{N}_{0}=\mathbb{N} \cup \{0\}$.

\begin{proposition}\label{psi_char}
Let $(T_{h}\phi)_{h \in H}$ be a Bessel sequence. Then for each $H^{\bot}$-periodic function $m \in L^{2}(U^{*})$ we have $m\hat{\phi}\in L^{2}(G^{*})$. Here $m$ is extended by $H^{\bot}$-periodicity to the function $m$ on $L^{2}(G^{*})$. Moreover, if $(T_{h}\phi)_{h \in H}$ is a frame for $\langle \phi \rangle$, then
\begin{equation} 
\langle \phi \rangle=\{f \in L^{2}(G):\hat{f}=m\hat{\phi}, \; m \in L^{2}(U^{*})\}
\end{equation}	
\end{proposition}	

\begin{proof}
 Since the system $\lbrace \chi(h,\cdot)\rbrace_{h \in H}$ is an orthonormal basis, we have 
\begin{equation*}
m(\omega)=\sum_{h \in H}\mu_{h}\overline{\chi(h,\omega)}
\end{equation*}
where $\mu=(\mu_{h})_{h \in H}\in l^{2}(\mathbb N_0)$ and
\begin{equation*}
\Vert m \Vert^{2}=\sum_{h \in H} \vert \mu_{h}\vert ^{2}
\end{equation*}

Let $g=\sum_{h \in H}\mu_hT_h \phi$. 
Then $\hat{g}(\omega)=m(\omega)\hat{\phi}(\omega) \in L^2(G^*).$

Further, $f \in L^2(G)$ such that $\hat{f}=m\hat{\phi}$, we have $\hat{f}=\sum_{h \in H}\mu_{h}\overline{\chi(h,\cdot)}\hat{\phi}$. 
Then by taking inverse Fourier transform on both sides,

  we obtain $f=\sum_{h \in H}\mu_{h}T_{h}\phi$. Thus, 
$$
\{f \in L^{2}(G):\hat{f}=m \hat{\phi}, m \in L^{2}(U^{*})\} \subseteq \langle \phi \rangle.
$$

Thus, for each $f \in \langle \phi \rangle$, there exists $\mu=(\mu_{h})_{h \in H} \in l^{2}(\mathbb N_0)$ such that 
$$
f = \sum_{h \in H}\mu_{h}T_{h}\phi.
$$
Applying Fourier transform we get,
\begin{equation*}
\hat{f}(\omega)=\sum_{h \in H}\mu_{h}\overline{\chi(h,\omega)}\hat{\phi}(\omega)=m(\omega)\hat{\phi}(\omega),
\end{equation*}
where $m(\omega)=\sum_{h \in H}\mu_{h}\overline{\chi(h,\omega)}$. Thus, 
$$
\langle \phi \rangle \subseteq \{f \in L^{2}(G):\hat{f}=m \hat{\phi}, m \in L^{2}(U^{*})\}.
$$
\end{proof}

\begin{remark}	
From here onwards we will identify Hilbert spaces $l^{2}(\mathbb N_0)$ and $L^{2}(U^{*})$ using the unitary operator $\mu=(\mu_{h})_{h \in H} \mapsto \sum_{h \in H}\mu_{h}\overline{\chi(h,\cdot)}$.
\end{remark}

\begin{remark}
For $\phi \in L^{2}(G)$, let the sequence $(T_{h}\phi)_{h \in H}$ be a Parseval frame for $\langle \phi \rangle$. Each function $m \in L^{2}(U^{*})$ such that $\hat{f}=m\hat{\phi}$, is called a filter for $f$. For $\hat{f}=m_{f}\hat{\phi}$, $\Vert m_{f}\Vert \leq \Vert m \Vert$. The function $m_{f}$ is called the \emph{minimal filter for $f$}. 
$$
[\hat{f},\hat{\phi}](\omega)=\sum_{h \in H}\langle f,T_{h}\phi \rangle \overline{\chi(h,\omega)}=m_f(\omega).
$$
\end{remark}

\begin{proposition}
Suppose that both $\Phi=\{\phi_{i}:i \in I\}$ and $\Psi=\{\psi_{j}:j \in J\}$ are Parseval frames which generate the same space, i.e $\langle\Phi \rangle=\langle\Psi \rangle$, then $\sum_{i \in I}[\hat{\phi}_{i},\hat{\phi}_{i}](\omega)=\sum_{j \in J}[\hat{\psi}_{j},\hat{\psi}_{j}](\omega)$ a.e.
\end{proposition}

\begin{proposition}\label{range}
Suppose that $\phi\in L^{2}(G)$ is such that the sequence $(T_{h}\phi)_{h \in H}$ is a Parseval frame for $\langle \phi \rangle$. Let $\Omega=\text{supp}(P_{\phi}) =\{\omega \in G^* : P_{\phi}(\omega) \neq 0\}$. Then 
\begin{equation}\label{200}
Range(\Lambda)= \{m \in L^{2}(U^{*}): m(\omega)=0, \text{  for a.e.  } \omega \notin \Omega\}
\end{equation}
In particular, for each $f \in \langle\phi \rangle$, the minimal filter $m_{f}$ is characterized among all filters for
$f$ by the property $m_{f}(\omega)=0$, for a.e. $\omega \notin \Omega$.
\end{proposition}

\begin{proof}

For each $m \in Range(\Lambda)$, 
$$
m(\omega)=\sum_{h \in H}\langle f,T_{h}\phi \rangle\overline{\chi(h,\omega)}. 
$$

Let $\omega \notin \Omega$. Then $\hat{\phi}(\omega \oplus h)=0$, for all $h \in H^{\bot}$. Thus $m(\omega)=0$, for $\omega \notin \Omega$.

Now, suppose that $m(\omega)=0$ for a.e. $\omega \notin \Omega$. Then we have to show that $m \in Range(\Lambda)=Ker(\Lambda^{*})^{\perp}$. Let $m=m_{1} \oplus m_{2}$, where $m_{1}\in Range(\Lambda)$ and $m_{2} \in Ker(\Lambda^{*})$. From above, we have $m_{1}(\omega)=0$, for all $\omega \notin \Omega$. Then $m_{2}(\omega)=0$ for a.e. $\omega \notin \Omega$. 

Since $m_{2}\in Ker(\Lambda^{*})$, which means $m_{2}(\omega)\hat{\phi}(\omega)=0$ a.e. By the $H^{\bot}$-periodicity of $m_{2}$, this implies that $m_{2}(\omega) \hat{\phi}(\omega \oplus h)=0$ a.e. for all $h \in H^{\bot}$. Now, if $\omega \in \Omega$, then $P_{\phi}(\omega) \neq 0$ and hence we must have $\hat{\phi}(\omega \oplus h)\neq 0$ for at least one $h$. Then from the previous equality we get that $m_{2}(\omega)=0$. Hence, $m_{2}(\omega)=0$ for a.e. $\omega$, and $m=m_{1}\in Range(\Lambda)$. The last statement follows from remark 2.
\end{proof}

\begin{theorem}\label{onb}
For $\phi \in L^{2}(G)$, the system $(T_{h}\phi)_{h \in H}$ is a parseval frame for $\langle\phi \rangle$ iff there exists a $H^{\bot}$-periodic set $\Omega \subseteq G^{*}$ such that $P_{\phi}=\chi_{\Omega}$ a.e. In particular, $(T_{h}\phi)_{h \in H}$ is an orthonormal basis for $\langle \phi \rangle$ iff $P_{\phi}(\omega)=1$, for a.e. $\omega \in G^{*}$.
\end{theorem}

For $f \in L^2(G)$, let $E_{f}=\{\omega \in G^{*}:P_{f}(\omega)>0\}$.
Let $\hat{f}_{\omega}=(\hat{f}(\omega \oplus h))_{h \in H^{\bot}}$ belongs to $l^2(\mathbb N_0)$, for a.e. $\omega \in U^{*}$. Let $S \subset L^{2}(G)$, then $\hat{S}_{\omega}=\{\hat{f}_{\omega}:f \in S\}$. Let $S$ be a shift-invariant subspace of $L^{2}(G)$, then $\eta(S)$ of $S$ is defined by $\eta(S)=\{\omega \in U^{*}:\hat{S}_{\omega}\neq \{0\}\}$ . For $f \in L^2(G)$, let $V_{f}=\{\omega \in U^{*}: \hat{f}_{\omega} \neq 0\}$.

\begin{proposition}
For the principal shift invariant space $\langle \phi \rangle$, we have
$$ 
\eta(\langle \phi\rangle)=V_{\phi}=E_{\phi}\cap U^{*}.
$$
\end{proposition}

\begin{proof}
For $f \in \langle \phi \rangle$, we have $\hat{f}(\omega)=m(\omega)\hat{\phi}(\omega)$. Then $\hat{f}_{\omega}=(\hat{f}(\omega \oplus h))_{h \in H^{\bot}}$.

Let $\omega \in \eta(\langle \phi \rangle)$, then there exists $f \in \langle \phi \rangle$ such that 
$$
\hat{f}_{\omega} \neq 0.
$$
Thus, $\hat{\phi}_{\omega} \neq 0$ implies that $\omega \in V_{\phi}$. Therefore $\eta(\langle \phi \rangle)\subseteq V_{\phi}$.

Further, if $\omega \in  V_{\phi}$, then $\hat{\phi}_{\omega}=(\hat{\phi}(\omega \oplus h))_{h \in H^{\bot}} \neq 0$. For any $H^{\bot}$-periodic non-zero function $m \in L^2(U^*)$, we have $(m(\omega)\hat{\phi}(\omega \oplus h))_{h \in H^{\bot}} \neq 0$. Hence $\omega \in \eta(\langle \phi \rangle)$ and $V_{\phi} \subseteq \eta(\langle \phi \rangle)$. This proves that $V_{\phi}=\eta(\langle\phi\rangle)$.

For $\omega \in V_{\phi}$, we have $V_{\phi} \subseteq E_{\phi} \cap U^*$. Similarly, $E_{\phi} \cap U^{*} \subseteq  V_{\phi}$. Thus, we have $\eta(\langle \phi \rangle)=V_{\phi}=E_{\phi} \cap U^{*}$. 
\end{proof}

\begin{corollary}
For the principal shift invariant space $\langle \phi \rangle$, the following are equivalent:
\begin{enumerate}
\item[(i)] $\{T_{h}\phi: h \in H\}$ is a Parseval frame of $\langle \phi \rangle$.
\item[(ii)] $\sum_{h \in H^{\bot}}\vert\hat{\phi}(\omega \oplus h)\vert^{2}=\chi_{E_{\phi}}(\omega)$ a.e.
\item[(iii)] $\sum_{h \in H^{\bot}}\vert\hat{\phi}(\omega \oplus h)\vert^{2}=\chi_{\eta(\langle \phi \rangle)\oplus H^{\bot}}(\omega)$ a.e.
\end{enumerate}
\end{corollary}

\begin{corollary}
Suppose that the principal shift invariant spaces generated by $\phi$ and $\psi$ are equal i.e $\langle \phi \rangle=\langle \psi \rangle$. Then, we have $E_{\phi}=E_{\psi}$.
\end{corollary}

\begin{proposition}
Let $\phi \in L^{2}(G)$ and assume that $\{T_{h}\phi\}_{h \in H}$ is a frame for $\langle \phi \rangle$. Then 
$ST_{h}=T_{h}S$ and $S^{-1}T_{h}=T_{h}S^{-1}$ on $\langle \phi \rangle$, for all $h \in H$, where $S=\Lambda^*\Lambda$ is the frame operator.
\end{proposition}

\begin{proof}
Given $f \in \langle \phi \rangle$ and $h \in H$, we have 
\begin{align*}
ST_{h}f & = \sum_{h' \in H}\langle T_{h}f,T_{h'}\phi \rangle T_{h'}\phi\\
        & = \sum_{h' \in H}\langle f,T_{h'}\phi \rangle T_{h'\oplus h}\phi\\
        & = T_{h}Sf.
\end{align*}
Similarly, we can proof the second part of this lemma.
\end{proof}

\begin{theorem}\label{bessel char1}
For every $\phi \in L^{2}(G)$, the following conditions are equivalent:
\begin{enumerate}
\item[(i)] The sequence $\{T_{h}\phi:h \in H\}$ forms a Bessel sequence in $L^{2}(G)$ with bound $D$.
\item[(ii)] The synthesis operator $\Lambda^{*}$ is well defined and $\Vert \Lambda^{*}\Vert \leq \sqrt{D}$.
\item[(iii)] $[\hat{\phi},\hat{\phi}](\omega)\leq D$ a.e. on $U^{*}$.
\end{enumerate}
\end{theorem}

\begin{proof}
The equivalence of first two statements is obvious.
	 
(ii) can be written as $\Vert \Lambda^{*}c \Vert^{2} \leq D\Vert c\Vert^{2}$, for every $c=(c_h)\in l^2(\mathbb N_0)$.  Now using the Parseval identity, we have $\Vert \sum_{h \in H}c_{h}\widehat{T_{h}\phi}\Vert^{2} \leq D\Vert c\Vert^{2}$. 

Now suppose that (iii) holds. Then for every $c=(c_h)\in l^2(\mathbb N_0)$
\begin{equation}\label{114}
\int_{U^{*}}\vert \sum_{h \in H}c_{h}\overline{ \chi(h,\omega)}\vert ^2[\hat{\phi},\hat{\phi}](\omega) \; d\omega \leq D \int_{U^{*}}\vert\sum_{h \in H}c_{h}\overline{\chi(h,\omega)}\vert ^{2} \;d\omega, 	\text{ for } c=(c_h)_{h \in H} \in l^{2}(\mathbb{N}_0)
\end{equation}
which gives
\begin{gather*}
\int_{U^{*}}\vert \sum_{h \in H}c_{h}\overline{\chi(h,\omega)}\vert^{2}\sum_{h' \in H^{\bot}}\vert \hat{\phi}(\omega \oplus h')\vert^{2} \; d\omega \leq 
\Vert\sum_{h \in H}c_{h}\widehat{T_{h}\phi}\Vert^{2} \leq D\Vert c \Vert^{2}
\end{gather*}

Thus, (iii) implies (ii). 

Conversely, suppose that (iii) does not hold. Then there exists a measurable set $E \subset U^{*}$ having positive measure such that 
\begin{equation}
[\hat{\phi},\hat{\phi}](\omega)> D \text{ a.e. on } E.
\end{equation}

Since $\chi_{E} \in L^{2}(U^{*})$, then from \eqref{114}, we get
$$
\int_{U^{*}}\vert \chi_{E}(\omega) \vert^{2}[\hat{\phi},\hat{\phi}](\omega)\;d\omega \leq D\int_{U^{*}}\vert \chi_{E}(\omega)\vert^{2}\;d\omega
$$ 
that implies $[\hat{\phi},\hat{\phi}](\omega) \leq D$ a.e. on $E$, which is a contradiction.	
\end{proof}

\begin{definition}\label{frame sequence}
A sequence $\{f_{j}\}_{j \in \mathbb N_0}$ is a \emph{frame sequence} in Hilbert space $H$ if the synthesis operator $T$ is well defined on $l^{2}(\mathbb{N}_{0})$ and there exist constants $0<C \leq D< \infty$ such that 
$$
C \Vert c \Vert^{2}\leq \Vert Tc \Vert^{2} \leq D\Vert c \Vert^{2} \text{ for } c \in (\text{Ker}(T))^{\perp}.
$$ 
\end{definition}

\begin{theorem}\label{frame seq char1}
The sequence $\{T_{h}\phi:h \in H\}$ in $L^{2}(G)$ is a frame sequence with frame bounds $C$ and $D$ iff 
\begin{equation}\label{115}
C \leq [\hat{\phi},\hat{\phi}](\omega) \leq D \text{ for a.e. $\omega$ in }U^{*}\setminus \mathcal{N}_{\phi},
\end{equation}
where $\mathcal{N}_{\phi}=\{\omega \in U^{*}:[\hat{\phi},\hat{\phi}](\omega)=0\}$.
\end{theorem}
	
\begin{proof}
From definition \ref{frame sequence}, we have $\{T_{h}\phi:h \in H\}$ is a frame sequence in $L^{2}(G)$ with frame bounds $C$ and $D$ iff the synthesis operator $\Lambda^*$ is well defined and
\begin{equation}\label{116}
C\Vert c \Vert^{2} \leq \Vert \Lambda^*c\Vert^{2} \leq D\Vert c\Vert ^{2} \text{ for } c \in (\text{Ker}(\Lambda^*))^{\perp}.
\end{equation}
If the right hand side of the inequality \eqref{116} holds then by the Theorem \ref{bessel char1}, it is equivalent to right hand side of inequality \eqref{115}. Therefore, we only need to prove the equivalence of 
\begin{equation}\label{117}
C\Vert c \Vert^{2} \leq \Vert \Lambda^*c\Vert ^{2} \text{ for }  c \in (\text{Ker}(\Lambda^*))^{\perp}
\end{equation}
and 
\begin{equation}\label{118}
C \leq  [\hat{\phi},\hat{\phi}](\omega) \text{ a.e on } U^{*} \setminus \mathcal{N}_{\phi},
\end{equation}

Now, 
\begin{equation*}
\Vert \Lambda^{*}c\Vert^{2}=\int_{U^{*}}\vert \sum_ {h \in H}c_{h}\overline{\chi(h,\omega)}\vert^{2} [\hat{\phi},\hat{\phi}](\omega) \; d\omega, \text{ for } c=(c_h) \in l^{2}(\mathbb{N}_{0}),
\end{equation*}

therefore
\begin{equation*}
\text{Ker}(\Lambda^{*})=\{(c_h) \in l^{2}(\mathbb{N}_{0}):\sum_{h \in H}c_{h}\overline{\chi(h,\omega)}=0 \text{ for a.e. } \omega \in U^{*} \setminus \mathcal{N}_{\phi}\}.
\end{equation*}
and
\begin{equation*}
(\text{Ker}(\Lambda^{*}))^{\perp}=\{(c_h) \in l^{2}(\mathbb{N}_{0}):\sum_{h \in H}c_{h}\overline{\chi(h,\omega)}=0 \text{ for a.e. } \omega \in \mathcal{N}_{\phi}\}.
\end{equation*}

Thus for $c \in (\text{Ker}(\Lambda^{*}))^{\perp}$, 
\begin{equation*}
\Vert c\Vert^{2}=\int_{U^{*}}\vert \sum_{h \in H}c_{h}\overline{\chi(h,\omega)}\vert^{2}d\omega=\int_{U^{*} \setminus \mathcal{N}_{\phi}}\vert\sum_{h \in H}c_{h}\overline{\chi(h,\omega)}\vert^{2}d\omega,
\end{equation*}
and 
\begin{equation*}
\Vert \Lambda^{*}c\Vert^{2}=\int_{U^{*} \setminus \mathcal{N}_{\phi}} \vert \sum_{h \in H}c_{h}\overline{\chi(h,\omega)} \vert^{2}[\hat{\phi},\hat{\phi}](\omega)d\omega.
\end{equation*}	
Thus by \eqref{117}, we get
\begin{equation}\label{119}
C \int_{U^{*} \setminus \mathcal{N}_{\phi}} \vert \sum_{h \in H}c_{h}\overline{\chi(h,\omega)} \vert ^{2}d\omega \leq \int_{U^{*} \setminus \mathcal{N}_{\phi}}\vert \sum_{h \in H}c_{h}\overline{\chi(h,\omega)}\vert^{2}[\hat{\phi},\hat{\phi}](\omega)d\omega.
\end{equation}	
Now to show that \eqref{119} is equivalent to \eqref{118}. Clearly, \eqref{118} implies \eqref{119}. To prove the converse part we can proceed by the same argument as in the proof of Theorem \ref{bessel char1}.
\end{proof}

\begin{corollary}
The sequence $\{ T_{h}\phi: h \in H\}$ is a Parseval frame in $L^{2}(G)$ iff 
$$
[\hat{\phi},\hat{\phi}](\omega)=\chi_{U^{*} \setminus \mathcal{N}_{\phi}}(\omega), \text{ for a.e. } \omega \in U^*,
$$
where $\mathcal{N}_{\phi}=\{\omega \in U^{*}:[\hat{\phi},\hat{\phi}](\omega)=0\}$.
\end{corollary}

\begin{lem}\label{conv1}
Let $\phi \in L^{2}(G)$ and assume that $\{T_{h}\phi\}_{h \in H}$ is a Bessel sequence. Let $\{c_{h}\}_{h \in H} \in l^{2}(\mathbb N_0)$. Then $\sum_{h \in H}c_{h}T_{h}\phi$ converges in $L^{2}(G)$ and $\sum_{h \in H}c_{h}\overline{\chi(h,\omega)}$ converges in $L^{2}(U^{*})$ and 
\begin{equation}\label{59}
\widehat{\sum_{h \in H}c_{h}T_{h}\phi}(\omega)=(\sum_{h \in H}c_{h}\overline{\chi(h,\omega)})\hat{\phi}(\omega).
\end{equation}
\end{lem}

\begin{proof}
Clearly, $\sum_{h \in H}c_{h}T_{h}\phi$ converges, then by \eqref{59} we have 
$\sum_{h \in H}c_{h}\overline{\chi(h,\omega)}$ converges in $L^{2}(U^{*})$. 
$$
\widehat{\sum_{h \in H}c_{h}T_{h}\phi(\omega)}=\sum_{h \in H}c_{h}\overline{\chi(h,\omega)}\hat{\phi}(\omega).
$$
\end{proof}

\begin{proposition}
Let $\phi \in L^{2}(G)$ and assume that $\{T_{h}\phi\}_{h \in H}$ is a frame for $\langle\phi \rangle$, with frame operator $S=\Lambda^*\Lambda$. Define a function $\theta$ via its Fourier transform by 
\begin{equation}\label{60}
\hat{\theta}(\omega)=\left\{\begin{matrix}
\frac{\hat{\phi}(\omega)}{P_{\phi}(\omega)} & \text{ if } \omega \in \Omega, \\ 
0 & \text{ if } \omega \notin \Omega,
\end{matrix}\right.
\end{equation}
where $\Omega=\text{supp}(P_{\phi})$. Then $\theta=S^{-1}\phi$.
\end{proposition}

\begin{proof}
The function 
$$
\zeta(\omega) = \left\{\begin{matrix}
\frac{1}{P_{\phi}(\omega)}&\text{ if } \omega \in \Omega \\ 
 0 & \text{ if } \omega \notin \Omega
\end{matrix}\right.$$
is $H^{\bot}$-periodic and its restriction to $U^{*}$ belongs to $L^{2}(U^{*})$. Thus, by Prop. \ref{psi_char} the function $\theta$ defined by \eqref{60} belongs to $\langle \phi \rangle$. Using the definition of frame operator, properties of the Fourier transform and Lemma \ref{conv1}, we have   
\begin{align*}
\widehat{S\theta}&=\widehat{\sum_{h \in H}\langle \theta, T_{h}\phi\rangle T_{h}\phi}\\
&=\sum_{h \in H}\langle \hat{\theta}, \widehat{T_{h}\phi}\rangle \widehat{T_{h}\phi}
\end{align*}
\begin{align}\label{61}
\widehat{S\theta}(\omega)&=(\sum_{h \in H}\langle \hat{\theta},\overline{\chi(h,\cdot)}\hat{\phi}\rangle \overline{\chi(h,\omega)})\hat{\phi}(\omega)
\end{align}

\begin{align*}
\langle \hat{\theta},\overline{\chi(h,\cdot)}\hat{\phi}\rangle &=\int_{G^{*}}\hat{\theta}(\omega)\overline{\hat{\phi}(\omega)}\chi(h, \omega)d\omega\\
&=\int_{U^{*}}\sum_{h' \in H^{\perp}}\frac{\vert\hat{\phi}(\omega \oplus h')\vert^{2}}{P_{\phi}(\omega \oplus h')}\chi_{\Omega}(\omega \oplus h')\chi(h,\omega)d\omega\\
&=\int_{U^{*}}\chi_{\Omega\cap U^{*}}(\omega)\chi(h,\omega)\; d\omega. 
\end{align*}

Therefore,
$$
\sum_{h \in H}\langle \hat{\theta},\overline{\chi(h,\cdot)}\hat{\phi}\rangle \overline{\chi(h,\omega)}=\chi_{\Omega \cap U^{*}}(\omega), \text{ for } \omega \in U^{*}.
$$
Since $\chi_{\Omega}$ is $H^{\perp}$-periodic, it follows that 
$$
\sum_{h \in H}\langle \hat{\theta},\overline{\chi(h,\cdot)}\hat{\phi}\rangle \overline{\chi(h,\omega)}=\chi_{\Omega}(\omega), \text{ for } \omega \in G^{*}.
$$
Note that $\chi_{\Omega}(\omega) \neq 0$ if $\hat{\phi}(\omega) \neq 0,$ now \eqref{61} implies that 
$$\widehat{S\theta}=\chi_{\Omega}\hat{\phi}=\hat{\phi}.$$
Therefore $S\theta=\phi$ and since $S$ is an invertible operator on $\langle \phi \rangle$ hence $\theta=S^{-1}\phi$.
\end{proof}

\begin{proposition}
Let $\phi \in L^{2}(G)$ and assume that $\{T_{h}\phi\}_{h \in H}$ is a frame sequence. Define the function $\phi^{*}$ via its Fourier transform by
\begin{equation}
\hat{\phi}^{*}(\omega)=\left\{\begin{matrix}
\hat{\phi}(\omega)P_{\phi}^{-1/2}(\omega), & \text{ if }\hat{\phi}(\omega) \neq 0, \\ 
 0, & \text{ if }\hat{\phi}(\omega) =0.
\end{matrix}\right.
\end{equation}
Then $\{T_{h}\phi^{*}\}_{h \in H}$ is a tight frame sequence, and $\langle\phi\rangle=\langle \phi^{*} \rangle$. 
\end{proposition}

\begin{proof}

Let $\omega \in U^{*}$. If $\hat{\phi}(\omega \oplus h)=0$, for all $h \in H^{\bot}$, then $P_{\phi^{*}}(\omega)=0$, so now assume that $\hat{\phi}(\omega \oplus h')\neq 0$ for some $h' \in H^{\bot}$. Then $P_{\phi^{*}}(\omega) \neq 0$.
Therefore, 
\begin{align*}
P_{\phi^{*}}(\omega)=\sum_{h \in H^{\bot}}\vert \hat{\phi^{*}}(\omega \oplus h)\vert^{2}&=\sum_{h \in H^{\bot}}\vert\hat{\phi}(\omega \oplus h)P_{\phi}^{-1/2}(\omega \oplus h)\vert^{2}\\
&=1 \text{ a.e. }
\end{align*}

Thus by Theorem \ref{frame seq char1} $\{T_{h}\phi^{*}\}_{h \in H}$ is a tight frame sequence. 

Note that 
$$
\langle \phi^* \rangle=\lbrace  \sum_{h \in H} c_hT_h\phi^* : (c_h)_{h \in H} \in l^2(\mathbb N_0)\rbrace.
$$
Taking the Fourier transform of the functions in $\langle \phi^* \rangle$, we get
$$
\lbrace \sum_{h\in H} c_h\overline{\chi(h,\cdot)}\hat{\phi}^*: (c_h)_{h \in H} \in l^2(\mathbb N_0)\rbrace=\lbrace \beta \sum_{h\in H} c_h\overline{\chi(h,\cdot)}\hat{\phi}: (c_h)_{h \in H} \in l^2(\mathbb N_0)\rbrace,
$$
where 
\begin{equation}
\beta(\omega)=\left\{\begin{matrix}
P_{\phi}^{-1/2}(\omega), & \text{ if }\hat{\phi}(\omega) \neq 0, \\ 
1/\sqrt{D}, & \text{ if }\hat{\phi}(\omega) =0,
\end{matrix}\right.
\end{equation}
and $C \leq P_{\phi}(\omega) \leq D$, for a.e. $\omega \in U^* \setminus \mathcal N_{\phi}$. Since $\beta$ is a bounded function, therefore
$$
\lbrace \beta \sum_{h\in H} c_h\overline{\chi(h,\cdot)}\hat{\phi}: (c_h)_{h \in H} \in l^2(\mathbb N_0)\rbrace = \lbrace \sum_{h\in H} c_h\overline{\chi(h,\cdot)}\hat{\phi}: (c_h)_{h \in H} \in l^2(\mathbb N_0)\rbrace.
$$

Thus, $\langle\phi\rangle=\langle \phi^{*} \rangle$.
\end{proof}

\section{Generalized dual frames}\label{sec4}

For the general Hilbert spaces we have the following lemma.

\begin{lem}\label{bessel1}
Let $\{f_{k}\}_{k=1}^{\infty}$ and $\{g_{k}\}_{k=1}^{\infty}$ be Bessel sequences in a Hilbert space $H$, and assume that
\begin{equation}
f_{j}=\sum_{k=1}^{\infty}\langle f_{j},g_{k}\rangle f_{k}, \quad \forall j \in \mathbb{N}.
\end{equation}
Then $\{f_{k}\}_{k=1}^{\infty}$ is a frame for $\overline{\text{span}}\{f_{k}\}_{k=1}^{\infty}$, and 
\begin{equation}
f=\sum_{k=1}^{\infty}\langle f,g_{k}\rangle f_{k}, \quad \forall f \in \overline{\text{span}}\{f_{k}\}_{k=1}^{\infty}.
\end{equation}
 \end{lem}

Here $\{f_{k}\}_{k=1}^{\infty}$ and $\{g_{k}\}_{k=1}^{\infty}$ are called the dual frames.

\begin{lem}\label{bessel2}
Let $\phi,\tilde{\phi} \in L^{2}(G)$. Assume that $\{T_{h}\phi\}_{h \in H}$ and $\{T_{h}\tilde{\phi}\}_{h \in H}$ are Bessel sequences. Then for all $f \in L^{2}(G)$,
$$
\widehat{\sum_{h \in H}\langle f,T_{h}\tilde{\phi}\rangle T_{h}\phi}(\omega)=\hat{\phi}(\omega) [\hat{f},\hat{\tilde{\phi}}](\omega).
$$
\end{lem}

\begin{proof}

\begin{align*}
\widehat{\sum_{h \in H}\langle f,T_{h}\tilde{\phi}\rangle T_{h}\phi(\omega)}&=\sum_{h \in H}\langle f,T_{h}\tilde{\phi}\rangle \widehat{T_{h}\phi(\omega)}\\
&=\hat{\phi}(\omega)\sum_{h \in H}\langle f,T_{h}\tilde{\phi}\rangle \overline{\chi(h, \omega)}
\end{align*}
Hence, we get 
 $$\widehat{\sum_{h \in H}\langle f,T_{h}\tilde{\phi}\rangle T_{h}\phi(\omega)}=\hat{\phi}(\omega)[\hat{f},\hat{\tilde{\phi}}](\omega).$$
\end{proof}

\begin{proposition}\label{bessel3}
For $\phi,\tilde{\phi} \in L^{2}(G)$, assume that $\{T_{h}\phi\}_{h \in H}$ and $\{T_{h}\tilde{\phi}\}_{h \in H}$ are Bessel sequences. Then the following are equivalent:
\begin{enumerate}
\item[(i)] $f=\sum_{h \in H}\langle f,T_{h}\tilde{\phi}\rangle T_{h}\phi, \forall f \in \langle \phi \rangle$
\item[(ii)] $\sum_{h \in H^{\perp}}\hat{\phi}(\omega \oplus h)\overline{\hat{\tilde{\phi}}(\omega \oplus h)}=1,$ a.e. on the set $\{\omega:P_{\phi}(\omega) \neq 0\}$.
\end{enumerate}
If the conditions are satisfied, then $\{T_{h}\phi\}_{h \in H}$ is a frame for $\langle \phi \rangle$. 
\end{proposition}

\begin{proof}
Assume that (i) holds. Then by Lemma \ref{bessel2} with $f=\phi$, 
$$\hat{\phi}(\omega)=\hat{\phi}(\omega) \sum_{h \in H^{\perp}}\hat{\phi}(\omega \oplus h)\overline{\hat{\tilde{\phi}}(\omega \oplus h)}.$$
Then for all $h' \in H^{\perp}$, 
$$\hat{\phi}(\omega \oplus h')=\hat{\phi}(\omega \oplus h') \sum_{h \in H^{\perp}}\hat{\phi}(\omega \oplus h)\overline{\hat{\tilde{\phi}}(\omega \oplus h)}.$$
Hence (ii) follows.
 
Now, assume that (ii) holds then by Lemma \ref{bessel2}, for any $h' \in H$ and $\omega$ such that $P_{\phi}(\omega) \neq 0$
\begin{align*}
\widehat{\sum_{h \in H}\langle T_{h'}\phi, T_{h}\tilde{\phi}\rangle T_{h}\phi(\omega)}
&=\hat{\phi}(\omega)\sum_{h^{*} \in H^{\perp}}\widehat{T_{h'}\phi}(\omega \oplus h^{*})\overline{\hat{\tilde{\phi}}(\omega \oplus h^{*})}\\
&=\hat\phi(\omega)\sum_{h^{*} \in H^{\perp}}\hat{\phi}(\omega \oplus h^{*})\overline{\hat{\tilde{\phi}}(\omega \oplus h^{*})}\overline{\chi(h',\omega)}\\
&=\widehat{T_{h'}\phi}(\omega)
\end{align*}
If $P_{\phi}(\omega)=0$, then the above holds trivially. Thus (i) follows by Lemma \ref{bessel1}.
\end{proof}

Let $I$ be a countable index set. For a given sequence of functions $\{\phi_{l}\}_{l \in I}$ in $L^{2}(G)$, let $\mathbb{W}=\overline{\text{span}}\{T_{h}\phi_{l}\}_{h \in H, l \in I}$.

\begin{theorem}\label{bessel4}
Let $\{\phi_{l}\}_{l \in I}, \{\tilde{\phi}_{l}\}_{l \in I} \subseteq L^{2}(G)$. Assume that $\{T_{h}\phi_{l}\}_{h \in H, l \in I}$ and $\{T_{h}\tilde{\phi}_{l}\}_{h \in H, l \in I}$ are Bessel sequences. Then the following are equivalent:
\begin{itemize}
\item[(i)] $f=\sum_{l \in I}\sum_{h \in H}\langle f,T_{h}\tilde{\phi}_{l}\rangle T_{h}\phi_{l}, \forall f \in \mathbb{W}$.
\item[(ii)] For all $n \in I$, 
  $$\hat{\phi}_{n}(\omega)=\sum_{l \in I}\hat{\phi}_{l}(\omega)(\sum_{h \in H^{\perp}}\hat{\phi}_{n}(\omega \oplus h)\overline{\hat{\tilde{\phi}}_{l}(\omega \oplus h)}),$$
\end{itemize}
holds for a.e. $\omega$.  
\end{theorem}

\begin{proof}
Assume that $(i)$ holds. Let $f=T_{h}\phi_{n}, h \in H, n \in I$, 

By $(i)$ and Lemma \ref{bessel2},
\begin{align*}
\hat{f}(\omega)=\widehat{T_{h}\phi_{n}(\omega)}&=\widehat{\sum_{l \in I}\sum_{h' \in H}\langle T_{h}\phi_{n},T_{h'}\tilde{\phi}_{l}\rangle T_{h'}\phi_{l}}(\omega)\\
&=\sum_{l \in I}\hat{\phi}_{l}(\omega) \sum_{h' \in H^{\perp}}\widehat{T_{h}\phi_{n}}(\omega \oplus h')\overline{\hat{\tilde{\phi}}_{l}(\omega \oplus h')}
\end{align*}

\begin{align*}
\hat{\phi}_{n}(\omega)\overline{\chi(h,\omega)}&=\sum_{l \in I}\hat{\phi}_{l}(\omega)\sum_{h' \in H^{\perp}}\hat{\phi}_{n}(\omega \oplus h')\overline{\chi(h,\omega \oplus h')}\overline{\hat{\tilde{\phi}}_{l}(\omega \oplus h')}\\
\text{Then },
\hat{\phi}_{n}(\omega)&=\sum_{l \in I}\hat{\phi}_{l}(\omega)\sum_{h' \in H^{\perp}}\hat{\phi}_{n}(\omega \oplus h')\overline{\hat{\tilde{\phi}}_{l}(\omega \oplus h')}
\end{align*}
This proves $(ii)$. On the other hand, if $(ii)$ holds then $(i)$ holds  for all $\{T_{h}\phi_{n}\}_{h \in H, n \in I}$ and hence for all $f \in \mathbb{W}$ by Lemma \ref{bessel1}. 
\end{proof}

\noindent \textbf{Note.} If the above conditions are satisfied, then $\{T_{h}\phi_{l}\}_{h \in H, l \in I}$ and $\{T_{h}\tilde{\phi}_{l}\}_{h \in H, l \in I}$ are called \emph{generalized dual frames}.

\begin{corollary}\label{510}
Let $\{\phi_{l}\}_{l \in I}, \{\tilde{\phi}_{l}\}_{l \in I} \subseteq L^{2}(G)$. Assume that $\{T_{h}\phi_{l}\}_{h \in H, l \in I}$ and $\{T_{h}\tilde{\phi}_{l}\}_{h \in H, l \in I}$ are Bessel sequences and the series 
\begin{equation}\label{54}
 \sum_{l \in I}\sum_{h \in H^{\perp}}\hat{\phi}_{l}(\omega)\hat{\phi}_{n}(\omega \oplus h)\overline{\hat{\tilde{\phi}}_{l}(\omega \oplus h)}
\end{equation}
converges unconditionally for all $n \in I$. Then the following are equivalent:
\begin{itemize}
\item[(i)] $f=\sum_{l \in I}\sum_{h \in H}\langle f,T_{h}\tilde{\phi}_{l}\rangle T_{h}\phi_{l}, \forall f \in \mathbb{W}$.
\item[(ii)] For all $n \in I$,  
 \begin{equation}\label{55}
  \hat{\phi}_{n}(\omega)=\sum_{h \in H^{\perp}}(\sum_{l \in I}\hat{\phi}_{l}(\omega)\overline{\hat{\tilde{\phi}}_{l}(\omega \oplus h)})\hat{\phi}_{n}(\omega \oplus h), \text{ holds for a.e. }\omega.
 \end{equation}
\end{itemize}
Condition $(ii)$ holds in particular, if 
\begin{equation}\label{56}
\sum_{l \in I}\hat{\phi}_{l}(\omega)\overline{\hat{\tilde{\phi}}_{l}(\omega \oplus h)}=\delta_{h,0} \text{ a.e. on } \{\omega:\sum_{l \in I}\vert\hat{\phi}_{l}(\omega)\vert^{2} \neq 0\}.
\end{equation}
\end{corollary}

\begin{proof}
The series in \eqref{54} is a reordering of the series on the right-hand side of $(ii)$ in Theorem \ref{bessel4}; this leads to the first part of the conclusion. We now prove that \eqref{55} holds for all $n \in I$, 
if \eqref{56} holds; in fact 
\begin{align*}
 \sum_{h \in H^{\perp}}(\sum_{l \in I}\hat{\phi}_{l}(\omega)\overline{\hat{\tilde{\phi}}_{l}(\omega \oplus h)})\hat{\phi}_{n}(\omega \oplus h)
 &=\left\{\begin{matrix}
 \hat{\phi}_{n}(\omega), & \text{ if } \sum_{l \in I}\vert\hat{\phi}_{l}(\omega)\vert^{2}\neq 0\\ 
 0, & \text{ if } \sum_{l \in I}\vert\hat{\phi}_{l}(\omega)\vert^{2}=0
\end{matrix}\right.\\
&=\hat{\phi}_{n}(\omega)
\end{align*}
\end{proof}

\begin{lem}
Assume that $\{T_{h}\phi_{l}\}_{h \in H, l \in I}$ and $\{T_{h}\tilde{\phi}_{l}\}_{h \in H, l \in I}$ are Bessel sequences and that $\sum_{l \in I}\vert\hat{\phi}_{l}(\omega)\vert<\infty$ a.e. Then the series in $\eqref{54}$ converges unconditionally for a.e. $\omega$. 
\end{lem}

\begin{proof}
If $D$ denotes a bound for the Bessel sequence $\{T_{h}\phi_{l}\}_{h \in H,l \in I}$, then $D$ is also a bound for $\{T_{h}\phi_{l}\}_{h \in H}$, for each $l \in I$. Then $\forall n$,
$$\sum_{h \in H^{\bot}}\vert\hat{\phi}_{n}(\omega \oplus h)\vert^{2} \leq D \quad \quad \text{ a.e. }$$
\begin{align*}
\sum_{l \in I}\sum_{h \in H^{\perp}}\vert\hat{\phi}_{l}(\omega)\hat{\phi}_{n}(\omega \oplus h)\overline{\hat{\tilde{\phi}}_{l}(\omega \oplus h)}\vert &=\sum_{l \in I}\vert\hat{\phi}_{l}(\omega)\vert\sum_{h \in H^{\perp}}\vert\hat{\phi}_{n}(\omega \oplus h)\overline{\hat{\tilde{\phi}}_{l}(\omega \oplus h)}\vert\\
&\leq \sum_{l \in I}\vert\hat{\phi}_{l}(\omega)\vert(\sum_{h \in H^{\perp}}\vert\hat{\phi}_{n}(\omega \oplus h)\vert^{2})^{\frac{1}{2}}(\sum_{h \in H^{\perp}}\vert \hat{\tilde{\phi}}_{l}(\omega \oplus h)\vert^{2})^{\frac{1}{2}}\\
&\leq C\sum_{l \in I}\vert\hat{\phi}_{l}(\omega)\vert, 
\end{align*}
where the constant $C$ depends on the bounds of $\{T_{h}\phi_{l}\}_{h \in H, l \in I}$ and $\{T_{h}\tilde{\phi}_{l}\}_{h \in H, l \in I}$. Thus the series \eqref{54} converges unconditionally if $\sum_{l \in I}\vert\hat{\phi}_{l}(\omega)\vert<\infty$ a.e.
\end{proof}

For $l \in \mathbb Z$, let $\phi_{l}(\omega)=D^{l}\phi(\omega)$ and $\tilde{\phi}_{l}(\omega)=D^{l}\tilde{\phi}(\omega)$, condition \eqref{55} can be written as 
\begin{equation}\label{57}
\hat{\phi}(B^{-n}\omega)=\sum_{h \in H^{\bot}}(\sum_{l \in \mathbb{Z}}p^{-l}\hat{\phi}(B^{-l}\omega)\overline{\hat{\tilde{\phi}}(B^{-l}(\omega \oplus h))})\hat{\phi}(B^{-n}(\omega \oplus h)) \text{ for a.e. }\omega
\end{equation}

\begin{proposition}
The condition \eqref{57} is satisfied for all $n \in \mathbb{Z}$ if and only if  
\begin{equation}\label{58}
\hat{\phi}(\omega)=\sum_{h \in H^{\perp}}(\sum_{m \in \mathbb{Z}}p^{-(m+n)}\hat{\phi}(B^{-m}\omega)\overline{\hat{\tilde{\phi}}(B^{-m}(\omega \oplus B^{-n}h))})\hat{\phi}(\omega \oplus B^{-n}h) \text{ for a.e. }\omega
\end{equation}
In particular, this condition is satisfied, if
$$\sum_{l \in \mathbb{Z}}p^{-l}\hat{\phi}(B^{-l}\omega)\overline{\hat{\tilde{\phi}}(B^{-l}(\omega \oplus B^{-n}h)}=p^n\delta_{h,0}, \text{ a.e.  on the set }\{\omega:\sum_{l \in \mathbb{Z}}\vert\hat{\phi}(B^{-l}\omega)\vert^{2} \neq 0\}.$$
\end{proposition}

\begin{proof}
From the equation \eqref{57}, 
\begin{align*}
\hat{\phi}(\omega)&=\sum_{h \in H^{\perp}}(\sum_{l \in \mathbb{Z}}p^{-l}\hat{\phi}(B^{-(l-n)}\omega)\overline{\hat{\tilde{\phi}}(B^{-(l-n)}(\omega \oplus B^{-n}h))})\hat{\phi}(\omega \oplus B^{-n}h)\\
&=\sum_{h \in H^{\perp}}(\sum_{m \in \mathbb{Z}}p^{-(m+n)}\hat{\phi}(B^{-m}\omega)\overline{\hat{\tilde{\phi}}(B^{-m}(\omega \oplus B^{-n}h))})\hat{\phi}(\omega \oplus B^{-n}h) 
\end{align*}
Hence \eqref{58} holds and similar to corollary \ref{510}, we get the last equation.
\end{proof}

\end{document}